\documentclass[12pt,a4paper]{amsart}
\usepackage{amssymb}
\usepackage{amsfonts}
\usepackage{amsmath}
\usepackage[mathscr]{eucal}

\usepackage{comment}

\newtheorem{thm}{Theorem}[section]
\newtheorem{prop}[thm]{Proposition}

\newtheorem{cor}[thm]{Corollary}

\newtheorem{rem}[thm]{Remark}

\numberwithin{equation}{section}

\def\GL{{\operatorname {GL}}}

\def\SL{{\operatorname{SL}}}

\def\PSL{{\operatorname {PSL}}}
\def\PGL{{\operatorname{PGL}}}

\def\Im{{\operatorname {Im}}}

\def\geq{\geqslant}

\def\ge{\geq}

\def\1{{\bold 1}}

\renewcommand{\a}{\alpha}

\newcommand{\f}{\varphi}

\newcommand{\s}{\sigma}

\newcommand{\CC}{\mathbb{C}}

\newcommand{\HH}{\mathbb{H}}

\newcommand{\RR}{\mathbb{R}}

\newcommand{\ZZ}{\mathbb{Z}}

\newcommand{\Res}{\operatorname{Res}}

\title[A remark on the existence of equivariant functions]{A remark on the existence of equivariant functions}

\author{Shingo Sugiyama}
\address{
	Faculty of Mathematics and Physics,
Institute of Science and Engineering,
Kanazawa University,
Kakumamachi, Kanazawa, Ishikawa, 920-1192, Japan
}
\email{s-sugiyama@se.kanazawa-u.ac.jp}

\subjclass[2020]{Primary 11F03; Secondary 34M05,
11F12}

\keywords{
Equivariant functions, vector-valued automorphic forms,
Schwarzian derivatives}

\begin{document}

\begin{abstract}
Let $\Gamma$ be a Fuchsian group
in $\SL_2(\RR)$.
In this note, we discuss the existence of $\rho$-equivariant functions for a two-dimensional representation $\rho$ of $\Gamma$.
This assertion was first stated by Saber and Sebbar in 2020, and
this note partially fills a gap of their statement by proving the assertion for a certain class of 
Fuchsian groups such as conjugates of subgroups of $\SL_2(\ZZ)$.
\end{abstract}

\maketitle

\section{Introduction}
\label{Introduction}

Let $\HH$ be the Poincar\'e upper-half plane.
Let $\Gamma$ be a Fuchsian group which means 
a discrete subgroup of $\SL_2(\RR)$.
Let $\rho$ be a two-dimensional representation of $\Gamma$,
i.e., a homomorphism $ \rho : \Gamma \rightarrow \GL_2(\CC)$.
A $\CC$-valued meromorphic function $h$ on $\HH$ is called a $\rho$-equivariant function (for $\Gamma$)
if $$h(\gamma z)=\rho(\gamma)h(z)$$
for all $\gamma \in \Gamma$ and $z \in \HH$ except for the poles of $h$,
where both $\gamma$ and $\rho(\gamma)$ act on complex numbers by linear transformation.
The notion of $\rho$-equivariant functions can be naturally introduced also when $\rho$ is replaced with
any of homomorphisms $\rho : \bar\Gamma \rightarrow \GL_2(\CC)$, $\rho : \Gamma \rightarrow \PGL_2(\CC)$ and
$\rho : \bar \Gamma \rightarrow \PGL_2(\CC)$,
where $\bar \Gamma$ is the subgroup of $\PSL_2(\RR)$ corresponding to $\Gamma$.

The notion of $\rho$-equivariant functions
was introduced by Saber and Sebbar \cite{SS2014},
which is the same as covariant functions by Kaneko
and Yoshida \cite{KanekoYoshida}.
It is a generalization of automorphic functions, just as automorphic functions on a Fuchsian group $\Gamma$ are examples of $\rho$-equivariant functions when $\rho(\gamma)=I_2$ for all $\gamma \in \Gamma$,
where $I_2$ denotes the two-by-two unit matrix.
The notion of $\rho$-equivariant functions
also generalizes equivariant functions studied in
\cite{SebbarSebbar}, \cite{ES1} and \cite{ES2},
which are meromorphic functions $h$ on $\HH$ satisfying $h(\gamma z)=\gamma h(z)$ for all $\gamma \in \Gamma$ and $z \in \HH$ except for the poles of $h$.
As a remarkable fact,
$\rho$-equivariant functions
are related to (meromorphic) automorphic forms of weight $4$
via the Schwarzian derivative. Here the Schwarzian derivative $\{h,z\}$ of a non-constant meromorphic function $h$ on a complex domain is defined as
$$\{h, z\} = \left(\frac{h''}{h'}\right)'-\frac{1}{2}\left(\frac{h''}{h'}\right)^2.$$
Let
$h$ be a non-constant meromorphic function on $\HH$
and $\Gamma$ a Fuchsian group.
Then, it is known that
the Schwarzian derivative $\{h,z\}$
is an automorphic form of weight $4$ on $\Gamma$
if and only if
$h$ is $\rho$-equivariant for a two-dimensional projective representation $\rho$ of $\Gamma$
(\cite[Proposition 3.1]{SS2020})\footnote{In \cite[Proposition 3.1]{SS2020}, $h$ should be non-constant.
Moreover, $\rho$ should be a projective representation from $\Gamma$ to $\PGL_2(\CC)$.}.
A $\rho$-equivariant function has been studied in the view point of
the Schwarzian derivative and
automorphic Schwarzian equations (see \cite{SS2020},
\cite{SS2021hypergeom}, \cite{SS2021} and \cite{SS2022}).

In this note, we discuss the problem on the existence of $\rho$-equivariant functions $h$ for any two-dimensional representation $\rho$ of any Fuchsian group. This problem is concerned with the difference between ``$\SL_2(\RR)$, $\GL_2(\CC)$'' and ``$\PSL_2(\RR)$, $\PGL_2(\CC)$''.
Because of
$(\pm I_2)h=I_2$, the action of $-I_2$ to $h$ seems negligible at first glance.
However, we must take care of the difference between $\SL_2(\RR)$ and $\PSL_2(\RR)$ if $\rho$ is a projective representation (homomorphism from $\Gamma$ to $\PGL_2(\CC)$).
Such a projective $\rho$ is not lifted to a homomorphism from $\Gamma$ to $\GL_2(\CC)$ in general. Indeed, $\rho$ is lifted to a homomorphism from the central extension of $\Gamma$ to $\GL_2(\CC)$.
Therefore many problems occur when we use theorems for projective representations in order to prove some properties for usual representations.

Our result on the existence of $\rho$-equivariant functions 
is stated as follows.

\begin{thm}\label{exist of equiv}
Let $\tilde{\Gamma}$ be a Fuchsian group.
Assume that there exist
a representation $\rho_0 : \tilde \Gamma \rightarrow \GL_2(\CC)$ of $\tilde \Gamma$
such that $\rho_0(-I_2) = I_2$ if $-I_2\in \tilde\Gamma$.
Further we assume the existence of a $\rho_0$-equivariant function
$h_0$ such that $\{h_0, z\}$ is holomorphic on $\HH$.
Then, for any Fuchsian group $\Gamma$ contained in $\tilde\Gamma$ and any representation $\rho : \Gamma \rightarrow \GL_2(\CC)$ of $\Gamma$ such that $\rho(-I_2) \in \CC^\times I_2$ if $-I_2\in \Gamma$,
there exists a $\rho$-equivariant function for $\Gamma$.
\end{thm}

Remark that the condition $\rho(-I_2) \in \CC^\times I_2$ is natural
as we see $h(z)=h(-I_2z)=\rho(-I_2)h(z)$ for any non-constant $\rho$-equivariant functions $h$, from which $\rho(-I_2) \in \CC^\times I_2$ holds.
We also note that the condition $\rho(-I_2) \in \CC^\times I_2$ immediately gives us $\rho(-I_2)=\pm I_2$.

A special case of Theorem \ref{exist of equiv} was given as
\cite[Theorem 7.2]{SS2017}, where $\rho(-I_2)=I_2$ was
imposed\footnote{
	Remark that $\Gamma$ in
	\cite{SS2017}
	is a subgroup of $\PSL_2(\RR)$ but not of $\SL_2(\RR)$.}
when $-I_2 \in \Gamma$. The assumption $\rho(-I_2)=I_2$
was essentially used in \cite[Theorem 7.2]{SS2017} since $\rho$-equivariant functions in \cite{SS2017}
were constructed by non-zero $\CC^2$-valued automorphic forms of weight $0$
with multiplier system $\rho$, where we note that the weight $0$ condition gives us $\rho(-I_2)=I_2$.

We show one example of problems due to the identification of usual representations with projective representations.
The result \cite[Theorem 7.2]{SS2017} was used
for $\Gamma=\SL_2(\ZZ)$ in \cite[p.1626]{SS2020},
where
the authors of \cite{SS2020} stated that any projective representation $\bar \rho : \PSL_2(\ZZ)\rightarrow \PGL_2(\CC)$
becomes a lift induced from a representation $\rho : \SL_2(\ZZ)\rightarrow \GL_2(\CC)$
and that this follows from the existence of a $\bar\rho$-equivariant function.
However, the existence of $\bar\rho$-equivariant functions
does not follow from \cite[Theorem 7.2]{SS2017}
since $\bar \rho$ is a projective representation but not a representation and
their argument works only for any representations $\rho : \PSL_2(\ZZ) \rightarrow \GL_2(\CC)$
but not for projective representations $\bar \rho : \PSL_2(\ZZ)\rightarrow \PGL_2(\CC)$.

Besides,
it was stated in \cite[p.554]{SS2021} that
$\rho$-equivariant functions always exist for any Fuchsian group $\Gamma$ in $\SL_2(\RR)$ and any projective representation $\rho : \Gamma \rightarrow \PGL_2(\CC)$ of $\Gamma$.
This statement does not follow from 
\cite[Theorem 7.2]{SS2017}
since $\rho$ is not a representation of $\Gamma$ as explained above.

Contrary to the previous result \cite[Theorem 7.2]{SS2017} where $\rho(-I_2)=I_2$ was imposed when $-I_2 \in \Gamma$,
Theorem \ref{exist of equiv} holds for all representations $\rho : \Gamma \rightarrow \GL_2(\CC)$
even when $\rho(-I_2) = -I_2$ under the assumption of the existence of $h_0$.
In \cite[\S2]{SS2021hypergeom}, it was stated that
$\rho$-equivariant functions always exist for any Fuchsian group $\Gamma$ in $\SL_2(\RR)$ and any representation $\rho : \Gamma \rightarrow \GL_2(\CC)$ of $\Gamma$. Theorem \ref{exist of equiv} justifies this statement partially.

As a corollary of Theorem \ref{exist of equiv}, we obtain the following by applying Klein's elliptic modular function $\lambda$ as $h_0$.

\begin{cor}\label{cor for SL_2(Z)}
	Let $\Gamma$ be any Fuchsian group such that $\Gamma \subset \s\SL_2(\ZZ)\s^{-1}$ for some $\s \in \SL_2(\RR)$.
	Then,
	for any representation $\rho : \Gamma \rightarrow \GL_2(\CC)$ of $\Gamma$
	such that $\rho(-I_2) \in \CC^\times I_2$ if $-I_2\in \Gamma$,
	there exists a $\rho$-equivariant function for $\Gamma$.
\end{cor}

\section{Proof of Theorem}

For $k \in \ZZ$ and a representation $\rho : \Gamma\rightarrow \GL_2(\CC)$ of a Fuchsian group $\Gamma$,
we say a $\CC^2$-valued meromorphic function $F$ on $\HH$ to be a $\CC^2$-valued automorphic form of weight $k$ and multiplier system $\rho$
if $F$ satisfies
$$F(\gamma z)=(cz+d)^k\rho(\gamma)F(z)$$
for all $\gamma=(\begin{smallmatrix}
	a & b \\ c & d
\end{smallmatrix}) \in \Gamma$ and all $z \in \HH$ except for the poles of $F$.
We do not impose conditions at the cusps of $\Gamma$ as in \cite[\S2]{SS2020}.
If a $\CC^2$-valued automorphic form $F=(\begin{smallmatrix}
f_1 \\ f_2
\end{smallmatrix})$
of weight $k$ and multiplier system $\rho$ satisfies $f_2\neq 0$, then
we can check that $\frac{f_1}{f_2}$
is a $\rho$-equivariant function.

By using the Schwarzian derivative, Saber and Sebbar \cite[Theorem 4.4]{SS2014} proved that,
for any two-dimensional representation $\rho$ of 
$\Gamma$
and any $\rho$-equivariant function $h$,
there exists a $\CC^2$-valued automorphic form
$F=(\begin{smallmatrix}
f_1 \\ f_2
\end{smallmatrix})$ of weight $-1$ and multiplier system $\rho$
such that $h=\frac{f_1}{f_2}$.
However, this statement is not true when $\rho(-I_2)=I_2$ since
there exist no non-zero $\CC^2$-valued automorphic forms of weight $-1$ and multiplier system $\rho$ in that case.

Furthermore, $h$ in \cite[Theorem 4.4]{SS2014} should be non-constant since the Schwarzian derivative of $h$ is used
in the proof.
If $h$ is constant, then the constant is a solution to
the equations $c z^2+(d-a)z-b=0$ for all $[\begin{smallmatrix}
	a & b \\ c & d \end{smallmatrix}] \in \Im\rho$.
This situation can happen
when $\Im\rho \subset \{\pm \delta^n \mid n \in \ZZ\}$ for some $\delta \in \GL_2(\CC)$, etc.
We modify \cite[Theorem 4.4]{SS2014} as follows.
\begin{prop}\label{exist of aut-1}
Let $\rho: \Gamma \rightarrow \GL_2(\CC)$
be a representation of a Fuchsian group $\Gamma$.
Let $h$ be a non-constant $\rho$-equivariant function such that
$\{h,z\}$ is holomorphic on $\HH$.
Then, there exists a representation $\rho' : \Gamma \rightarrow \GL_2(\CC)$
and a $\CC^2$-valued automorphic form $F=(\begin{smallmatrix}
f_1\\f_2
\end{smallmatrix})$ of weight $-1$ and multiplier system $\rho'$ such that
$f_1$ and $f_2$ are linearly independent and $h=\frac{f_1}{f_2}$.
In particular, $\rho$ equals $\chi\rho'$ for some character $\chi$ of $\Gamma$.
\end{prop}
For the proof of Proposition \ref{exist of aut-1}, we correct 
\cite[Theorem 3.3]{SS2014}
as follows.
\begin{prop}\label{global solution}
Let $D$ be a simply connected domain in $\CC$.
Let $h$ be a non-constant meromorphic function on $D$.
Assume that $g(z):=\{h, z\}$ is holomorphic on $D$.
Then, a square root $\sqrt{h'}$ of $h'$ is defined as a meromorphic function on $D$.
Moreover, $y''+\frac{1}{2}gy=0$ has two linearly independent holomorphic solutions on $D$ given by
$f_1=\frac{h}{\sqrt{h'}}$ and $f_2=\frac{1}{\sqrt{h'}}$.
\end{prop}
\begin{proof}
In the proof of \cite[Theorem 3.3]{SS2014},
the patching of local solutions $(K_i, L_i)$ on $U_i$ is not justified
since the equality $\a_i\a_j^{-1}=\a_W$ is not true.
This equality should be $\a_i\a_j^{-1}=\lambda_{ij}\a_W$ for some $\lambda_{ij}\in\CC^\times$. Thus (3.2) in \cite[Theorem 3.3]{SS2014} is not true. Moreover, the case where $D=\CC-\{0\}$ and $h=-\frac{1}{2z^2}$ is a counterexample of \cite[Theorem 3.3]{SS2014}.
In that case, we have $g=\{h,z\}=-\frac{3}{2z^2}$  and two fixed branches $z^{-1/2}$ and $z^{3/2}$ are linearly independent local solutions of $y''+\frac{1}{2}gy=0$ on a simply connected domain in $\CC-\{0\}$. These solutions are not analytically continued to $\CC-\{0\}$.

For the proof of the assertion, we refer to \cite[Theorem 3.3 (2)]{SS2020} on the explicit formula of two linearly independent solutions on $\HH$.
However, the proof of \cite[Theorem 3.3 (2)]{SS2020} should be also corrected since the meromorphy of $\sqrt{h'}$ is not proved by merely taking the principal branch of the square root. We need to prove that the orders of all poles of $h'$ are even.
We correct the proof of \cite[Theorem 3.3 (2)]{SS2020} as follows.

First we prove that $h'$ is non-vanishing everywhere on $D$.
If $h'(z_0)=0$ holds at some $z_0\in D$, then $\{h, z\}$ has a double pole at $z_0$. Indeed, if we put $h'(z)=(z-z_0)^n p(z)$ for a function $p$ with $p(z_0)\neq 0$ and $n\ge1$, we have
\begin{align}\label{explicit0}
\{h, z\}= -\frac{n(n+2)}{2(z-z_0)^2}-\frac{np'(z)}{(z-z_0)p(z)}+\frac{2p(z)p''(z)-3p'(z)^2}{2p(z)^2}
\end{align}
by a direct computation 
(cf.\ \cite[pp.38--39]{Yasukawa}). This contradicts the holomorphy of $\{h, z\}$.

Next we prove that every point in $D$ is a regular point or a simple pole of $h$.
If $z_0 \in D$ is a pole of $h$ of order $n\ge 2$,
then $\{h,z\}$ has a double pole at $z_0$. Indeed,
$1/h$ has a zero of order $n$ at $z_0$.
When $(1/h)'=(z-z_0)^{n-1}p(z)$ for a function $p$ with $p(z_0)\neq 0$, the same computation as \eqref{explicit0} leads us to \begin{align}\label{explicit}
\{h, z\}=\{1/h, z\}=-\frac{(n-1)(n+1)}{2(z-z_0)^2}-\frac{(n-1)p'(z)}{(z-z_0)p(z)}+\frac{2p(z)p''(z)-3p'(z)^2}{2p(z)^2}.
\end{align}
Hence $z_0$ is a double pole of $\{h, z\}$.
This contradicts the holomorphy of $\{h,z\}$.
We remark that \eqref{explicit} is valid for $n=1$. Therefore $h$ may have a simple pole since
$\{h,z\}$ is holomorphic at $z_0$ when $n=1$
by \eqref{explicit}.

For introducing $\sqrt{h'}$,
we use an elementary method of complex analysis (cf.\ \cite[Lemma 3.7]{Yasukawa}).
Fix a regular point $z_0 \in D$ of $h$ (or equivalently, of $h'$) and define a function $G$
by $$G(z):=\sqrt{h'(z_0)}\exp\left({\frac{1}{2}\int_{L_z}\frac{h''(\zeta)}{h'(\zeta)}d\zeta}\right)$$
for $z \in D-P_{h}$, where $P_h$ is the set of the poles of $h$, 
$\sqrt{h'(z_0)}$ is a fixed square root of $h'(z_0)$,
and $L_z$ is a fixed
smooth Jordan curve from $z_0$ to $z$ not passing through the poles of $h$.
Then $G(z)$ is independent of the choice of $L_z$.
Indeed,
when $L'_z$ is another smooth Jordan curve with the same property as $L_z$,
the argument principle gives us
$$\int_{L_z}\frac{h''(\zeta)}{h'(\zeta)}d\zeta -\int_{L_z'}\frac{h''(\zeta)}{h'(\zeta)}d\zeta=\pm2\pi\sqrt{-1}\sum_{a}\Res_{\zeta=a}\frac{h''(\zeta)}{h'(\zeta)}
=\pm2\pi\sqrt{-1}\sum_{a}(-2),$$
where $a$ runs over all poles of $h'$ in the bounded domain whose boundary is $L_z \cup L_z'$.
Here we use the assumption that $D$ is simply connected and the formula $\Res_{\zeta=a}\frac{h''(\zeta)}{h'(\zeta)}=-2$
since any singular point of $h'$ is its double pole.
Thus $G(z)$ is well-defined.

If we set $\varphi=G^2$, we can check $\varphi(z_0)=h'(z_0)$ and $\varphi'=2GG'=\varphi\frac{h''}{h'}$.
Hence we obtain $\varphi=h'$, i.e., $G^2=h'$.
Moreover, $G$ is a meromorphic function on $D$ which is regular at every regular point of $h$, and every pole of $h$ is a simple pole of $G$.
By the consideration so far,
the proof of all desired properties of $G$ is completed.

Finally, $f_1:=\frac{h}{G}$ and $f_2:=\frac{1}{G}$
are holomorphic on $D$ with the aid of the properties of $G$.
Furthermore, the linear independence of $f_1$ and $f_2$ is clear since $h$ is non-constant.
By using
$\varphi'= \varphi\frac{h''}{h'}$, $G'=\frac{\varphi}{2G}\frac{h''}{h'}=\frac{G}{2}\frac{h''}{h'}$
and
$G''=\frac{G}{2}(\frac{h''}{h'})'+\frac{G'}{2}\frac{h''}{h'}
=\frac{G}{2}(\frac{h''}{h'})'+\frac{G}{4}(\frac{h''}{h'})^2$,
we obtain $f_1''=-\frac{1}{2}gf_1$
and $f_2''=-\frac{1}{2}gf_2$.
\end{proof}

\begin{proof}[Proof of Proposition \ref{exist of aut-1}] We prove the assertion by correcting the argument in \cite[Theorem 4.4]{SS2014}.
For a given non-constant $\rho$-equivariant function $h$,
set $g=\{h,z\}$.
Then $g$ is an automorphic form of weight $4$ on $\Gamma$.
Note that $g$ is holomorphic on $\HH$ by the assumption.
By Proposition \ref{global solution},
the differential equation $y''+\frac{1}{2}gy=0$ has two linearly independent
holomorphic solutions $f_1$ and $f_2$ on $\HH$ such that $h=\frac{f_1}{f_2}$.

By \cite[Corollary 4.3]{SS2014},
the function $F
:=(\begin{smallmatrix}
	f_1\\f_2
\end{smallmatrix})$
is a $\CC^2$-valued automorphic form of weight $-1$ and multiplier system $\rho'$,
where $\rho' : \Gamma \rightarrow \GL_2(\CC)$ is the representation of $\Gamma$ given by
$$\left(\begin{smallmatrix}
	(cz+d)f_1(\gamma z)\\(cz+d)f_2(\gamma z)
\end{smallmatrix}\right)=\rho'(\gamma)\left(\begin{smallmatrix}
f_1(z)\\f_2(z)
\end{smallmatrix}\right), \qquad z \in \HH, \ \gamma \in \Gamma$$
(cf.\ \cite[Corollary 4.3]{SS2014}). Note $\rho'(-I_2)=-I_2$ by definition.
Fix any $\gamma \in \Gamma$. Since $h=\frac{f_1}{f_2}$ is both $\rho$-equivariant and $\rho'$-equivariant, we have
$\rho(\gamma)h(z)=\rho'(\gamma)h(z)$ for any $\gamma$ and any $z$.
As $h$ is non-constant and meromorphic,
$h$ takes three distinct values and hence $\rho(\gamma)$ equals $\rho'(\gamma)$ as a linear transformation.
Thus there exists $\chi(\gamma) \in \CC^\times$ such that
$\rho(\gamma)=\chi(\gamma)\rho'(\gamma)$.
We can check easily that $\chi$ is a character of $\Gamma$.
\end{proof}

\medskip
By using a sheaf cohomology,
we can show the existence of $\CC^2$-valued automorphic forms of weight $0$ by \cite[Theorem 6.2]{SS2017}, where the group $\Gamma$ in \cite[Theorem 6.2]{SS2017} is a subgroup of $\PSL_2(\RR)$ but not of $\SL_2(\RR)$.
By noting this, we have the following.
\begin{thm}[Theorem 7.2 in \cite{SS2017}]\label{exist of auto0}
Let $\Gamma$ be a Fuchsian group in $\SL_2(\RR)$
and
$\rho : \Gamma \rightarrow \GL_2(\CC)$ a representation of $\Gamma$
such that $\rho(-I_2)=I_2$ if $-I_2 \in \Gamma$.
Then there exists a $\rho$-equivariant function.
\end{thm}

\begin{prop}\label{exist of chi}
	Let $\Gamma$ be a Fuchsian group containing $-I_2$.
	Assume the existence of a representation 
$\rho_0 : \Gamma \rightarrow \GL_2(\CC)$ such that $\rho_0(-I_2)=I_2$.
We also assume the existence of a $\rho_0$-equivariant function
	$h_0$ such that $\{h_0, z\}$ is holomorphic on $\HH$.
	Then there exists a character $\chi$ of $\Gamma$
	such that $\chi(-I_2)=-1$.
\end{prop}
\begin{proof}
By Proposition \ref{exist of aut-1} for $h_0$, there exists a representation $\rho' : \Gamma \rightarrow \GL_2(\CC)$ of $\Gamma$
	and a $\CC^2$-valued automorphic form $F=(\begin{smallmatrix}f_1 \\ f_2\end{smallmatrix})$ of weight $-1$ and multiplier system $\rho'$
	such that $h_{0}=\frac{f_1}{f_2}$.
In particular, we have $\rho_0 = \chi \rho'$ for some character $\chi$ of $\Gamma$.
Here we can take $\rho'$ such that $\rho'(-I_2)=-I_2$ by the construction of $\rho'$ in the proof of Proposition \ref{exist of aut-1}. Hence we obtain $\chi(-I_2)I_2=\chi(-I_2)\rho_0(-I_2)=\rho'(-I_2)=-I_2$. This completes the proof.
\end{proof}

\begin{proof}[{Proof of Theorem \ref{exist of equiv}}]
We may assume $-I_2 \in \tilde\Gamma$ and $\rho(-I_2)=-I_2$, by Theorem
\ref{exist of auto0}.
Then we take a character $\chi$
of $\tilde\Gamma$ such that $\chi(-I_2)=-1$ by Proposition \ref{exist of chi}.
The restriction of $\chi$ to $\Gamma$ is denoted by $\chi_\Gamma$.
Then $\chi_{\Gamma}\,\rho$ satisfies $\chi_{\Gamma}\,\rho(-I_2)=I_2$,
which leads us to
the existence of
a $\chi_{\Gamma}\,\rho$-equivariant function
by Theorem \ref{exist of auto0}.
This function is also $\rho$-equivariant.
\end{proof}

\begin{proof}
[{Proof of Corollary \ref{cor for SL_2(Z)}}]
Klein's elliptic modular function $\lambda$ is a Hauptmodul
for $\Gamma(2)$, where $\Gamma(2)$ is the principal congruence subgroup of level $2$.
By \cite[\S 6]{SS2020}, $\lambda$ is a $\rho_0$-equivariant function
for $\SL_2(\ZZ)$. Here $\rho_0$ is a two-dimensional representation of $\SL_2(\ZZ)$ given by
$\rho_0(\begin{smallmatrix}
1&1\\0&1
\end{smallmatrix})=(\begin{smallmatrix}
1&0\\1&-1
\end{smallmatrix})$ and $
\rho_0(\begin{smallmatrix}
	0&-1\\1&0
\end{smallmatrix}) = (\begin{smallmatrix}
-1&1\\0&1
\end{smallmatrix}).$
We remark
$\rho_0(-I_2)=\rho_0(\begin{smallmatrix}
	0&-1\\1&0
\end{smallmatrix})^2=(\begin{smallmatrix}
-1&1\\0&1
\end{smallmatrix})^2=I_2.$
Moreover the equality $\{\lambda,z\}=\frac{\pi^2}{2}E_4$ holds,
where $E_4$ is the Eisenstein series of weight $4$ and level $1$ (see \cite[Proposition 5.2]{McKaySebbar}\footnote{The Schwarzian derivative in \cite{McKaySebbar} is the twice of ours.}).
By Theorem \ref{exist of equiv} for $\tilde \Gamma=\SL_2(\ZZ)$ and $h_0=\lambda$, we obtain the corollary.
\end{proof}

	If $\tilde\Gamma$ is a Fuchsian group of the first kind and of genus $0$
	with no elliptic elements, then a Hauptmodul $h_0$ for $\tilde\Gamma$
	is locally univalent on $\HH$ and thus $\{h_0, z\}$
	is holomorphic on $\HH$ (cf.\ \cite[Proposition 6.1]{McKaySebbar}).
	Explicit examples of the Schwarzian derivatives of Hauptmoduln
	are treated for $\Gamma_0(N)$ in \cite{McKaySebbar}
	and for $\Gamma(N)$ in \cite{SS2020}.

\begin{rem}\label{grp theory}
Let $\Gamma$ be a Fuchsian group containing $-I_2$ and
let $[\Gamma, \Gamma]$ be the commutator subgroup of $\Gamma$.
If $\Gamma$
is assumed to satisfy $-I_2 \notin [\Gamma, \Gamma]$,
then we can prove the existence of a character $\chi : \Gamma \rightarrow \CC^\times$ such that $\chi(-I_2)=-1$ group-theoretically.
Indeed, the subgroup $H$ of $\Gamma /[\Gamma, \Gamma]$
generated by $-I_2[\Gamma,\Gamma]$ is of order two.
Thus we can take a non-trivial character $\chi_0$ of $H$.
By the Pontrjagin duality, $\chi_0$ is lifted to a character $\chi$ of $\Gamma/[\Gamma,\Gamma]$, which is regarded as a character of $\Gamma$.
As $\chi_0$ is non-trivial, we have $\chi(-I_2)=-1$.

We can verify $-I_2 \notin [\SL_2(\ZZ), \SL_2(\ZZ)]$ by \cite[Theorem 1.3.1]{Rankin}.
The case of $\Gamma=\s\SL_2(\ZZ)\s^{-1}$ for some $\s \in \SL_2(\RR)$ is similarly treated.
\end{rem}

\section*{Acknowledgements}
The author was supported by
Grant-in-Aid for Young Scientists (20K14298).


\end{document}